\documentclass[a4paper]{amsart}

\usepackage[utf8]{inputenc}
\usepackage[T1]{fontenc}

\usepackage{amsthm, amssymb, amsmath, amsfonts, mathrsfs}

\usepackage[ocgcolorlinks, colorlinks=true, pdfstartview=FitV, linkcolor=blue, citecolor=blue, urlcolor=blue, pagebackref=false]{hyperref}

\usepackage{microtype}

\newtheorem{thm}{Theorem}[section]
\newtheorem{prop}[thm]{Proposition}
\newtheorem{lem}[thm]{Lemma}

\theoremstyle{remark}
\newtheorem{rem}[thm]{Remark}

\def\les{\lesssim}

\renewcommand{\subset}{\subseteq}

\newcommand{\E}{\mathbb{E}}

\newcommand{\cE}{\mathcal{E}}
\newcommand{\cG}{\mathcal{G}}

\newcommand{\M}{\mathcal{M}}
\newcommand{\W}{\mathcal{W}}

\newcommand{\R}{\mathbb{R}}
\newcommand{\C}{\mathcal{C}}

\renewcommand{\P}{\mathbb{P}}

\newcommand{\eps}{\varepsilon}

\numberwithin{equation}{section}

\title[Fluctuations in stochastic homogenization]{A central limit theorem for fluctuations in one dimensional stochastic homogenization}
\author{Yu Gu}

\address[Yu Gu]{Department of Mathematics, Building 380, Stanford University, Stanford, CA, 94305, USA}

\begin{document}
\begin{abstract}

In this paper, we analyze the random fluctuations in a one dimensional stochastic homogenization problem and prove a central limit result, i.e., the first order fluctuations can be described by a Gaussian process that solves an SPDE with additive spatial white noise. Using a probabilistic approach, we obtain a precise error decomposition up to the first order, which helps to decompose the limiting Gaussian process, with one of the components corresponding to the corrector obtained by a formal two scale expansion.

\bigskip


\medskip

\end{abstract}
\maketitle
%
%
%
%
%
%
%
%
\section{Introduction}

The equation we are interested in is 
\begin{equation}
\partial_t u_\eps=\frac12\partial_x \tilde{a}(\frac{x}{\eps},\omega)\partial_x u_\eps, \ \ t>0, x\in \R,
\label{eq:mainEq}
\end{equation}
with an initial condition $u_\eps(0,x)=f(x)\in \C_c^\infty(\R)$. Here $\tilde{a}$ is a smooth stationary random field defined on some probability space $(\Omega,\mathcal{F},\P)$, and satisfies
\[
\lambda \leq \tilde{a}(x,\omega)\leq 1
\]
for some $\lambda\in (0,1)$ and $x\in \R, \omega\in \Omega$. The standard homogenization result shows that $u_\eps\to u_{\hom}$ with 
\begin{equation}
\partial_t u_{\hom}=\frac12\bar{a}\partial_x^2u_{\hom},
\label{eq:homoEq}
\end{equation}
and the homogenization constant $\bar{a}$ is the harmonic mean of $\tilde{a}$:
\[
\bar{a}^{-1}=\E\{\tilde{a}^{-1}\}.
\]


The goal is to analyze the first order fluctuations, i.e., if the homogenization is viewed as a law of large numbers type result, we are interested in a central limit theorem (CLT) here. The same question has been addressed for the boundary value problem 
\begin{equation}
\label{eq:bvp}
-\frac{d}{dx}\tilde{a}(\frac{x}{\eps},\omega)\frac{d}{dx}u_\eps=f, \mbox{ with } x\in (0,1),
\end{equation}
and $u_\eps(0)=0, u_\eps(1)=1$, under different assumptions on the correlation properties of $\tilde{a}$ 
\cite{BP-AA-99,BGMP-AA-08,GB-JDE-12}. It was shown in \cite[Theorem 3.1]{BP-AA-99} that if $\tilde{a}$ satisfies certain mixing assumption, a CLT holds:
\[
\frac{u_\eps(x)-u_{\hom}(x)}{\sqrt{\eps}}\Rightarrow \int_0^1 F(x,y)dB_y
\]
in $\C([0,1])$, where $F(x,y)$ is deterministic and $B_y$ is a standard Brownian motion. The analysis used the explicit solution to \eqref{eq:bvp} and reduced the problem to the weak convergence of oscillatory random integrals. We ``revisit'' the problem on the whole space for the following reasons:

(i) It seems the approaches in \cite{BP-AA-99,BGMP-AA-08,GB-JDE-12} fails to work for \eqref{eq:mainEq} since we do not have an explicit solution when the problem is posed on the whole space, so a different method needs to be developed.

(ii) It is not clear how the boundary condition affects the asymptotic behavior of the rescaled fluctuations in \eqref{eq:bvp}, and the Wiener integral $\int_0^1 F(x,y)dB_y$ is not naturally related to the corrector obtained through the formal two scale expansion that is used extensively in homogenization. In this paper we are looking for an error decomposition that is in parallel to the formal expansion and also indicates clearly how each term contributes to the limiting Gaussian distribution.

(iii) It was shown recently when $d\geq 3$ that a pointwise formal two scale expansion holds \cite[Theorem 2.3]{gu2014pointwise}: for fixed $(t,x)$, 
\begin{equation}
u_\eps(x)=u_{\hom}(x)+\eps \nabla u_{\hom}(x)\cdot \tilde{\phi}(\frac{x}{\eps})+o(\eps),
\label{eq:loflu}
\end{equation}
where $\tilde{\phi}$ is the mean-zero stationary corrector and $o(\eps)/\eps\to 0$ in $L^1(\Omega)$. It can be seen from \eqref{eq:loflu} that the \emph{local} (pointwise) fluctuation is not necessarily Gaussian since $\tilde{\phi}$ is not Gaussian. On the other hand, the large scale central limit results are proved for the \emph{global} fluctuations of the solutions and correctors \cite{gu2015scaling,mourrat2015scaling}
\begin{eqnarray}
&&\frac{1}{\eps^{d/2}}\int_{\R^d}(u_\eps(x)-\E\{u_\eps(x)\})g(x)dx\Rightarrow N(0,\sigma_1^2),\\
&&\frac{1}{\eps^{d/2}}\int_{\R^d}\eps \nabla u_{\hom}(x)\cdot \tilde{\phi}(\frac{x}{\eps})g(x)dx\Rightarrow N(0,\sigma_2^2).
\end{eqnarray}
It turns out that $\sigma_1\neq \sigma_2$, so the corrector which represents the local fluctuation does not capture the global fluctuation! Mathematically it is not surprising since the $o(\eps)$ in \eqref{eq:loflu} could contribute on the level of $\eps^{d/2}$ when $d\geq 3$. From a practical point of view, it is important that we can extract the right term from $o(\eps)$ such that together with the corrector they represent the global fluctuations of the solutions. To understand the mechanism better, we start from the simpler setting $d=1$, where the local and global fluctuations are known to be described by a single Gaussian field on the level of $\sqrt{\eps}$. We expect the error decomposition and its relation to the corrector in low dimension to shed light on the situation in high dimensions.

\vskip 5mm

Quantitative aspects of stochastic homogenization of divergence form operator has witnessed a lot of progress recently, from both analytic and probabilistic points of view \cite{GO-AP-11,gloria2013quantification,gloria2014optimal,gloria2014quantitative,mourrat2014kantorovich}. Our approach falls into the more probabilistic side: we use the probabilistic representation of the solutions to \eqref{eq:mainEq} and quantify the weak convergence of an underlying diffusion in random environment. The main ingredients in our analysis consist of the Kipnis-Varadhan's method \cite{KV-CMP-86,KLO-SP-12} applied to reversible diffusion in random environment and the quantitative martingale CLT developed in \cite{mourrat2014kantorovich,gu2015fluctuations} to extract the first order error in the martingale convergence. We also rely heavily on the work \cite{iftimie2008homogenization}, where the authors analyzed the asymptotics of 
\begin{equation}
\partial_t u_\eps=\frac12\partial_x \tilde{a}(\frac{x}{\eps},\omega)\partial_x u_\eps+\frac{1}{\sqrt{\eps}}\tilde{c}(\frac{x}{\eps},\omega)u_\eps,
\label{eq:ipp}
\end{equation}
i.e., \eqref{eq:mainEq} with a large highly oscillatory random potential. It turns out that a part of the error in our analysis of $u_\eps-u_{\hom}$ solves \eqref{eq:ipp} with an additive rather than multiplicative potential. By following a similar argument, we obtain a limiting SPDE (for this part of the error) with additive white noise (which is a Gaussian process), while the limit of \eqref{eq:ipp} is an equation with multiplicative white noise. 

The Kipnis-Varadhan's method decomposes the underlying diffusion process as a small remainder plus a martingale which converges to the limit. One of our main contributions is to combine the errors coming from the remainder and the martingale convergence, which are three Gaussian processes, and write the sum as the solution to one SPDE with additive noise. On one hand, this justifies rigorously in the simpler setting $d=1$ the heuristics presented in \cite{gu2015scaling}. More precisely, if we assume that in \eqref{eq:mainEq}, the fluctuations of the coefficient $\tilde{a}$ around its homogenization limit $\bar{a}$ can be described by some mean zero, strongly mixing processes, denoted by $\tilde{V}$ in \eqref{eq:defv} below, and on the large scale, $\tilde{V}$ can be replaced by some spatial white noise $\dot{\W}$ (after proper rescaling), then \eqref{eq:mainEq} may be rewritten as
\begin{equation}
\partial_t u_\eps\approx \frac12\partial_x(\bar{a}+\sqrt{\eps}\dot{\W})\partial_xu_\eps,
\end{equation}
so the error $v_\eps=\eps^{-1/2}(u_\eps-u_{\hom})$ should solve
\begin{equation}
\partial_t v_\eps\approx\frac12\bar{a}\partial_x^2 v_\eps+\frac12\partial_x(\partial_xu_{\hom}\dot{\W}).
\label{eq:heueq}
\end{equation}
Indeed, we obtain the limiting fluctuation described by an equation of the form \eqref{eq:heueq} in Theorem~\ref{t:mainth} below. For more detailed discussions of the heuristics in high dimensions, we refer to \cite[Section 1]{gu2015scaling}. On the other hand, the three Gaussian processes obtained in the limit provide a natural decomposition of the limiting SPDE (on the level of equations), which is in parallel to the martingale decomposition of the underlying diffusion in random environment (on the level of  stochastic processes), and this helps us to see the role played by the corrector more clearly. It is not clear whether \eqref{eq:heueq} can be obtained through some PDE arguments.

\vskip 5mm

Here are some notations used throughout the paper. The expectation in $(\Omega,\mathcal{F},\P)$ is denoted by $\E$. When averaging with respect to some independent Brownian motion $B,W$, we use $\E_B,\E_W$. The normal distribution with mean $\mu$ and variance $\sigma^2$ is denoted by $N(\mu,\sigma^2)$, and the density function of $N(0,t)$ is denoted by $q_t(x)=(\sqrt{2\pi t})^{-1}e^{-|x|^2/2t}$. We write $a\les b$ when $a\leq Cb$ for some constant $C$ independent of $t,x,\eps$.

The rest of the paper is organized as follows. We present the setup and main results in Section~\ref{s:re}. The error decomposition is discussed in Section~\ref{s:err}, and weak convergence results are obtained in Section~\ref{s:wc}. In Section~\ref{s:dis} we finish the proof of the main result and compare it with high dimensions $d\geq 3$. Some technical lemmas are left in Section~\ref{s:lem}.

\section{Setup and main results}
\label{s:re}

We first assume there is a group of measure-preserving, ergodic transformation $\{\tau_x,x\in\R\}$ associated with the probability space $(\Omega,\mathcal{F},\P)$, then the coefficient field $\tilde{a}(x,\omega)$ is defined by
\[
\tilde{a}(x,\omega)=a(\tau_x\omega)
\]
 for some $a:\Omega\to [\lambda,1]$. We further assume it is smooth and of finite range dependence:
 
 (i) $\tilde{a}(x,\omega)$ has $\C^2$ sample paths whose first and second order derivatives are uniformly bounded in $(x,\omega)$.
 
(ii) For any two sets $A,B\subset \R$, if $\mathrm{dist}(A,B)\geq 1$, then $\mathcal{F}_A=\sigma(\tilde{a}(x,\omega): x\in A)$ is independent of $\mathcal{F}_B=\sigma(\tilde{a}(x,\omega):x\in B)$.

\begin{rem}
The finite range of dependence can be replaced by some mixing condition, e.g., the $\phi-$mixing used in \cite{iftimie2008homogenization}.
\end{rem}

Besides the coefficient field $\tilde{a}(x,\omega)$, another important random field in our analysis is 
\begin{equation}
\tilde{V}(x,\omega)=\frac{1}{\tilde{a}(x,\omega)}-\frac{1}{\bar{a}},
\label{eq:defv}
\end{equation}
which may be seen as the fluctuations of the homogenization constant. It is clear that $\tilde{V}$ is of finite range dependence, and its covariance function is given by $R(x)=\E\{\tilde{V}(0,\omega)\tilde{V}(x,\omega)\}$ and the power spectrum is 
\begin{equation}
\hat{R}(\xi)=\int_\R R(x)e^{-i\xi x}dx.
\end{equation}


The following is our main result:
\begin{thm}
Let $v_\eps=\eps^{-1/2}(u_\eps-u_{\hom})$ and $v$ solves
\begin{equation}
\partial_t v=\frac12\bar{a}\partial_x^2 v-\frac12\hat{R}(0)^\frac12\bar{a}^2 \partial_x (\partial_x u_{\hom}\dot{\W}), \mbox{ with } v(0,x)=0,
\label{eq:limitspde}
\end{equation}
where $\dot{\W}$ is spatial white noise, then as $\eps\to 0$, $v_\eps\Rightarrow v$ in the following sense:

(i) as a process in $(t,x)$, the finite dimensional distributions converge,

(ii) for any test function $g\in \C_c^\infty(\R)$, $\int_\R v_\eps(t,x)g(x)dx\Rightarrow \int_\R v(t,x)g(x)dx$ in distribution.
\label{t:mainth}
\end{thm}

It turns out the Gaussian process $v(t,x)$ is a superposition of three Gaussian processes, and one of them takes the form $\partial_x u_{\hom}(t,x)\W(x)$, which corresponds to the corrector obtained through a formal two scale expansion. Here 
\[
\W(x):=\int_0^x \dot{W}(y)dy
\]
 is a two-sided Brownian motion.

\subsection{Diffusion in random environment}

Our starting point to prove Theorem~\ref{t:mainth} is a probabilistic representation. For every fixed $\omega\in\Omega, x\in\R$ and $\eps>0$, we define the underlying diffusion in random environment by the It\^o's SDE:
\begin{equation}
dX_t=\tilde{b}(X_t,\omega)dt+\tilde{\sigma}(X_t,\omega)dB_t, \mbox{ with } X_0=x/\eps,
\label{eq:sde}
\end{equation}
where 
\[
\tilde{b}(x,\omega)=\frac12\tilde{a}'(x,\omega), \mbox{  } \tilde{\sigma}(x,\omega)=\tilde{a}^\frac12(x,\omega).
\]
The driving Brownian motion $B_t$ is built on another probability space $(\Sigma,\mathcal{A},\P_B)$. It is straightforward to check that $\eps X_{t/\eps^2}$ is a Markov process starting from $x$ with the generator $L^\omega=\frac12\partial_x\tilde{a}(x/\eps,\omega)\partial_x$, so the solution to \eqref{eq:mainEq} can be written as
\[
u_\eps(t,x)=\E_B\{ f(\eps X_{t/\eps^2})\},
\]
where $\E_B$ denotes the expectation in $(\Sigma,\mathcal{A},\P_B)$. 

It can be shown that $\eps X_{t/\eps^2}$ converges in distribution to $x+\bar{\sigma}W_t$ for some Brownian motion $W_t$ starting from the origin and $\bar{\sigma}=\sqrt{\bar{a}}$, so 
\[
u_\eps(t,x)=\E_B\{f(\eps X_{t/\eps^2})\}\to \E_W\{ f(x+\bar{\sigma}W_t)\}= u_{\hom}(t,x)
\]
in probability. It is clear that to further get the first order fluctuations $v_\eps=\eps^{-1/2}(u_\eps-u_{\hom})$, we need to quantify the weak convergence of $\eps X_{t/\eps^2}\Rightarrow x+\bar{\sigma}W_t$ up to the first order.

We define an environmental process by $\omega_t=\tau_{X_t}\omega$, and it satisfies the following properties \cite[Proposition 9.8]{KLO-SP-12}:
\begin{prop}
$(\omega_s)_{s\geq 0}$ is a Markov process that is reversible and ergodic with respect to the measure $\P$.
\end{prop}

In Sections~\ref{s:err} and \ref{s:wc}, we will only show that for fixed $(t,x)$, $v_\eps(t,x)\Rightarrow v(t,x)$ in distribution. To simplify the presentation,  we will shift the random environment $\omega$ by $x/\eps$ without changing the distribution of $u_\eps(t,x)$ (for fixed $(t,x)$), thus in the following we will assume
\[
u_\eps(t,x)=\E_B\{ f(x+\eps X_{t/\eps^2})\}
\]
with $X_t$ solving \eqref{eq:sde} but starting from the origin:
\begin{equation}
dX_t=\tilde{b}(X_t,\omega)dt+\tilde{\sigma}(X_t,\omega)dB_t, \mbox{ with } X_0=0.
\end{equation}
The convergence of finite dimensional distributions and the global weak convergence are discussed in Section~\ref{s:dis}.

To simplify the notations, we will omit the dependence on $\omega$ from now on.

\section{Quenched invariance principle and error decomposition}
\label{s:err}

To quantify the weak convergence, we first present a proof of $\eps X_{t/\eps^2}\Rightarrow \bar{\sigma}W_t$, where the diffusion in random environment is decomposed as a remainder plus a martingale, and the speed of weak convergence hinges on how small the remainder is and how ``close'' in distribution the martingale is to the limiting Brownian motion.

By the uniform ellipticity condition, we have the following standard heat kernel estimates \cite{stroock1988diffusion} which will be used extensively in our analysis.

(i) The density function $q_\eps(t,x)$ of $\eps X_{t/\eps^2}$ satisfies
\begin{equation}
q_\eps(t,x)\les \frac{1}{\sqrt{t}}e^{-c|x|^2/t}
\label{eq:heatk}
\end{equation}
uniformly in $\eps,\omega$ for some $c>0$.

(ii) For any $t>0$,
\begin{equation}
\P_B\{\sup_{s\in [0,t]}|\eps X_{s/\eps^2}|\geq M\}\les e^{-cM^2/t}
\label{eq:heatkp}
\end{equation}
uniformly in $\eps,\omega$ for some $c>0$.

The following result is classical. For the sake of convenience, we present the proof here.

\begin{prop}
For almost every realization $\omega\in\Omega$, $\eps X_{t/\eps^2}\Rightarrow \bar{\sigma}W_t$ in $\C(\R_+)$.
\label{p:qiv}
\end{prop}

\emph{Proof of Proposition~\ref{p:qiv}}. First we write the SDE as the integral equation
\[
\eps X_{t/\eps^2}=\eps \int_0^{t/\eps^2}\tilde{b}(X_s)ds+\eps\int_0^{t/\eps^2}\tilde{\sigma}(X_s)dB_s
\]
and by solving a corrector equation
\begin{equation}
-\frac12\frac{d}{dx} \tilde{a}\frac{d}{dx}\tilde{\phi}=\tilde{b}=\frac12\frac{d}{dx}\tilde{a},
\label{eq:coreq}
\end{equation}
and applying It\^o's formula, we have 
\begin{equation}
\eps X_{t/\eps^2}=R_t^\eps+M_t^\eps,
\label{eq:made}
\end{equation}
with 
\[
R_t^\eps=-\eps \tilde{\phi}(X_{t/\eps^2})+\eps \tilde{\phi}(X_0),
\]
and
\[
M_t^\eps=\bar{a}\eps\int_0^{t/\eps^2}\frac{1}{\tilde{\sigma}(X_s)}dB_s.
\]
Here $\tilde{\phi}$ satisfies 
\[
\tilde{\phi}'(x)=\bar{a}\left(\frac{1}{\tilde{a}(x)}-\frac{1}{\bar{a}}\right)=\bar{a}\tilde{V}(x),
\]
and we choose 
\begin{equation}
\tilde{\phi}(x)=\bar{a}\int_0^x \tilde{V}(y)dy.
\label{eq:defphi}
\end{equation}

For the remainder, we have 
\[
\begin{aligned}
\sup_{t\in [0,T]}|R_t^\eps|\leq & \sup_{t\in [0,T]}\eps |\tilde{\phi}(\eps X_{t/\eps^2}/\eps)|\\
\leq &\sup_{t\in [0,T]}\eps|\tilde{\phi}(\eps X_{t/\eps^2}/\eps)|1_{\sup_{t\in [0,T]}|\eps X_{t/\eps^2}|\leq M}\\
&+\sup_{t\in [0,T]}\eps|\tilde{\phi}(\eps X_{t/\eps^2}/\eps)|1_{\sup_{t\in [0,T]}|\eps X_{t/\eps^2}|> M}
\end{aligned}
\]
for any constant $M>0$. By ergodicity and the fact that $\E\{\tilde{V}\}=0$, we have for almost every realization $\omega\in\Omega$, $\tilde{\phi}(x)/x\to 0$ as $|x|\to \infty$, i.e., the corrector has a sublinear growth. This implies
\[
\sup_{t\in [0,T]}\eps|\tilde{\phi}(\eps X_{t/\eps^2}/\eps)|1_{\sup_{t\in [0,T]}|\eps X_{t/\eps^2}|\leq M}\to 0
\]
as $\eps\to 0$. For the second part, we use the sublinear growth to obtain
\[
\eps|\tilde{\phi}(\eps X_{t/\eps^2}/\eps)|1_{\sup_{t\in [0,T]}|\eps X_{t/\eps^2}|> M}\les |\eps X_{t/\eps^2}|1_{\sup_{t\in [0,T]}|\eps X_{t/\eps^2}|> M},
\]
so we only need to show that
\[
\E_B\{\sup_{t\in[0,T]}|\eps X_{t/\eps^2}|1_{\sup_{t\in [0,T]}|\eps X_{t/\eps^2}|> M}\}\to 0
\]
uniformly in $\eps$ as $M\to\infty$, but this comes from \eqref{eq:heatkp}. To summarize, we have shown that for almost every $\omega$, the remainder $R_t^\eps\to 0$ in $\C(\R_+)$.

For the martingale part, the quadratic variation can be written as
\[
\langle M^\eps\rangle_t=\bar{a}^2\eps^2\int_0^{t/\eps^2}\frac{1}{a(\tau_{X_s}\omega)}ds,
\]
so by ergodicity of $\tau_{X_s}\omega$ (it is independent of $\eps$ since we have shifted by environment), we have that for almost every $\omega$, $\langle M^\eps\rangle_t\to \bar{a}t$ almost surely as $\eps\to0$, and this implies $M_t^\eps\Rightarrow \bar{\sigma}W_t$ in $\C(\R_+)$. The proof is complete.

\vskip 3mm

Now we can decompose the error in homogenization according to the martingale decomposition of $\eps X_{t/\eps^2}$. By \eqref{eq:made}, we write
\[
u_\eps(t,x)-u_{\hom}(t,x)=\E_B\{f(x+R_t^\eps+M_t^\eps)\}-\E_W\{f(x+\bar{\sigma}W_t)\}:=\cE_1+\cE_2
\] 
with 
\[
\cE_1=\E_B\{f(x+R_t^\eps+M_t^\eps)\}-\E_B\{f(x+M_t^\eps)\},
\]
and
\[
\cE_2=\E_B\{f(x+M_t^\eps)\}-\E_W\{f(x+\bar{\sigma}W_t)\}.
\]
The main contribution from $\cE_1$ can be extracted by a Taylor expansion, i.e., we have
\[
|\cE_1-\E_B\{f'(x+M_t^\eps)R_t^\eps\}|\les \E_B\{|R_t^\eps|^2\}.
\]
The main contribution from $\cE_2$ can be extracted by a quantitative martingale CLT \cite[Proposition 3.2]{gu2015fluctuations}:
\[
\begin{aligned}
|\cE_2-\frac12\E_B\{f''(x+M_t^\eps)(\langle M^\eps\rangle_t-\bar{a}t)\}|\les &\E_B\{|\langle M^\eps\rangle_t-\bar{a}t|^{\frac32}\}\\
\leq & (\E_B\{|\langle M^\eps\rangle_t-\bar{a}t|^2\})^{\frac34}.
\end{aligned}
\]

Now we define 
\[
v_{1,\eps}(t,x)=\E_B\{f'(x+M_t^\eps)R_t^\eps\}, \ \ v_{2,\eps}(t,x)=\frac12\E_B\{f''(x+M_t^\eps)(\langle M^\eps\rangle_t-\bar{a}t)\},
\]
By Lemmas~\ref{lem:resm} and \ref{lem:masm} below, we have
\begin{equation}
\E\{|u_\eps(t,x)-u_{\hom}(t,x)-v_{1,\eps}(t,x)-v_{2,\eps}(t,x)|\}\ll \sqrt{\eps},
\end{equation}
therefore, to analyze the asymptotic distribution of $\eps^{-1/2}(u_\eps-u_{\hom})$, we only need to consider $v_{1,\eps}+v_{2,\eps}$.

\begin{lem}
$\E\E_B\{|R_t^\eps|^2\}\les \eps\sqrt{t} $
\label{lem:resm}
\end{lem}

\begin{lem}
$\E\E_B\{|\langle M^\eps\rangle_t-\bar{a}t|^2\}\les \eps t^{\frac32}$
\label{lem:masm}
\end{lem}

\begin{proof}[Proof of Lemma~\ref{lem:resm}]
Since $R_t^\eps=-\eps\tilde{\phi}(X_{t/\eps^2})+\eps\tilde{\phi}(X_0)$ and $\tilde{\phi}(x)=\bar{a}\int_0^x \tilde{V}(y)dy$,
we only need to consider $\eps\tilde{\phi}(X_{t/\eps^2})$. If we write 
\[
\E_B\{|\eps \tilde{\phi}(X_{t/\eps^2})|^2\}=\int_{\R} |\eps\tilde{\phi}(\frac{x}{\eps})|^2q_\eps(t,x)dx,
\]
with $q_\eps(t,x)$ the density function of $\eps X_{t/\eps^2}$ (which depends on the random realization $\omega$), by the heat kernel bound \eqref{eq:heatk} we have
\[
\E_B\{|\eps \tilde{\phi}(X_{t/\eps^2})|^2\}\les \int_\R |\eps\tilde{\phi}(\frac{x}{\eps})|^2\frac{1}{\sqrt{t}}e^{-c|x|^2/t}dx
\]
for every $\omega$. By Lemma~\ref{lem:mmpp}, $\E\{|\tilde{\phi}(x)|^2\}\les |x|$, so we can take $\E$ on both sides of the above expression to obtain 
\[
\E\E_B\{|R_t^\eps|^2\}\les \eps  \int_\R |x|\frac{1}{\sqrt{t}}e^{-c|x|^2/t}dx\les\eps \sqrt{t}.
\]
\end{proof}

\begin{proof}[Proof of Lemma~\ref{lem:masm}]
First, we write
\begin{equation}
\langle M^\eps\rangle_t-\bar{a}t=\bar{a}^2\eps^2\int_0^{t/\eps^2}\left(\frac{1}{\tilde{a}(X_s)}-\frac{1}{\bar{a}}\right)ds=\bar{a}^2\eps^2\int_0^{t/\eps^2}\tilde{V}(X_s)ds,
\end{equation}
where we recall $\tilde{V}=\tilde{a}^{-1}-\bar{a}^{-1}$. Since $\tilde{V}$ has mean zero, by ergodic theorem, we have $\langle M^\eps\rangle_t-\bar{a}t\to 0$ as $\eps\to 0$, but to quantify how small it is, we need to apply a martingale decomposition again, in the same spirit as for $\tilde{b}$. 

Let $\tilde{\psi}$ satisfy
\begin{equation}
-\frac12\frac{d}{dx}\tilde{a}\frac{d}{dx}\tilde{\psi}=\tilde{V},
\label{eq:defpsi}
\end{equation}
then 
\[
\eps^2\int_0^{t/\eps^2}\tilde{V}(X_s)ds=-\eps^2\tilde{\psi}(X_{t/\eps^2})+\eps^2\tilde{\psi}(X_0)+\eps^2\int_0^{t/\eps^2}\tilde{\psi}'(X_s)\tilde{\sigma}(X_s)dB_s.
\]
Since $\tilde{\phi}(x)=\bar{a}\int_0^x\tilde{V}(y)dy$, we can choose 
\begin{equation}
\tilde{\psi}(x)=-\frac{2}{\bar{a}}\int_0^x \frac{\tilde{\phi}(y)}{\tilde{a}(y)}dy.
\label{eq:defpsi1}
\end{equation}
By Lemma~\ref{lem:mmpp} we have $\E\{|\tilde{\psi}(x)|^2\}\les |x|^3$, and we follow the same discussion as in the proof of Lemma~\ref{lem:resm}. For $\eps^2\tilde{\psi}(X_{t/\eps^2})$ we have
\[
\E_B\{|\eps^2\tilde{\psi}(X_{t/\eps^2})|^2\}=\int_\R \eps^4|\tilde{\psi}(\frac{x}{\eps})|^2q_\eps(t,x)dx,
\]
with $q_\eps(t,x)$ the density of $\eps X_{t/\eps^2}$. Applying the heat kernel bound \eqref{eq:heatk} and taking $\E$, we conclude that
\[
\E\E_B\{|\eps^2\tilde{\psi}(X_{t/\eps^2})|^2\}\les \eps \int_\R |x|^3\frac{1}{\sqrt{t}}e^{-c|x|^2/t}dx \les \eps t^{\frac32}.
\]
Since $X_0=0$, we have $\eps^2\tilde{\psi}(X_0)=0$. For the martingale term, we have
\[
\begin{aligned}
\E_B\{|\eps^2\int_0^{t/\eps^2}\tilde{\psi}'(X_s)\tilde{\sigma}(X_s)dB_s|^2\}=&\eps^4\int_0^{t/\eps^2}\E_B\{|\tilde{\psi}'(X_s)\tilde{\sigma}(X_s)|^2\}ds\\
\les & \eps^4\int_0^{t/\eps^2} \E_B\{|\tilde{\phi}(X_s)|^2\}ds.
\end{aligned}
\]
By a similar discussion, we have
\[
\begin{aligned}
\eps^4\int_0^{t/\eps^2} \E\E_B\{|\tilde{\phi}(X_s)|^2\}ds=&\eps^2\int_0^t\E\E_B\{|\tilde{\phi}(\eps X_{s/\eps^2}/\eps)|^2\}ds\\
\les &\eps^2\int_0^{t}\int_\R\frac{|x|}{\eps}\frac{1}{\sqrt{s}}e^{-c|x|^2/s}dxds\les \eps t^{\frac32},
\end{aligned}
\]
which completes the proof.
\end{proof}

\section{Weak convergence}
\label{s:wc}

Now we consider $v_{1,\eps},v_{2,\eps}$ and prove the weak convergence of $\eps^{-1/2}(v_{1,\eps}+v_{2,\eps})$. 

Let us recall that
\begin{eqnarray}
&&v_{1,\eps}(t,x)=\E_B\{f'(x+M_t^\eps)R_t^\eps\}\\
&&v_{2,\eps}(t,x)=\frac12\E_B\{f''(x+M_t^\eps)(\langle M^\eps\rangle_t-\bar{a}t)\}
\end{eqnarray}
and give a heuristic explanation of what we may expect from the weak convergence of $\eps^{-1/2}(v_{1,\eps}+v_{2,\eps})$.

For $v_{1,\eps}$, we can write 
\[
\frac{R_t^\eps}{\sqrt{\eps}}=-\sqrt{\eps}\tilde{\phi}(X_{t/\eps^2})=-\sqrt{\eps}\tilde{\phi}(\frac{\eps X_{t/\eps^2}}{\eps}).
\]
On one hand, since $\tilde{\phi}(x)=\bar{a}\int_0^x \tilde{V}(y)dy$ and $\tilde{V}$ is a mean-zero stationary random process with finite range of dependence, by a classical functional central limit theorem \cite[pages 178,179]{billingsley1} we have
\[
\sqrt{\eps}\tilde{\phi}(\frac{x}{\eps})\Rightarrow \bar{c}\W(x)
\]
weakly in $\C(\R)$, where 
\begin{equation}
\bar{c}=\hat{R}(0)^\frac12\bar{a},
\end{equation}
and $\W(x)$ is a two-sided Brownian motion with $\W(0)=0$, i.e.,
\[
\W(x)=\left\{\begin{array}{ll}
\W_1(x) & x\geq 0,\\
\W_2(-x) & x<0,
\end{array}
\right.
\] 
where $\W_1,\W_2$ are independent Brownian motions starting from the origin. On the other hand, by Proposition~\ref{p:qiv}, we have 
\[
\eps X_{t/\eps^2}\Rightarrow \bar{\sigma} W_t
\]
 in $\C(\R_+)$ for almost every $\omega$. Apparently $\W$ and $W$ are independent since they live in $(\Omega,\mathcal{F},\P)$ and $(\Sigma,\mathcal{A},\P_B)$ respectively, so we expect
\[
\frac{R_t^\eps}{\sqrt{\eps}}=-\sqrt{\eps}\tilde{\phi}(\frac{\eps X_{t/\eps^2}}{\eps})\Rightarrow -\bar{c}\W(\bar{\sigma}W_t)
\]
in distribution.

For $v_{2,\eps}$, we have
\[
\frac{\langle M^\eps\rangle_t-\bar{a}t}{\sqrt{\eps}}=\bar{a}^2 \eps^\frac32\int_0^{t/\eps^2}\tilde{V}(X_s)ds=\bar{a}^2\frac{1}{\sqrt{\eps}}\int_0^t \tilde{V}(\frac{\eps X_{s/\eps^2}}{\eps})ds,
\]
and since
\[
\frac{1}{\sqrt{\eps}}\int_0^x\tilde{V}(\frac{y}{\eps})dy\Rightarrow \hat{R}(0)^\frac12\W(x),
\]
we can formally write 
\[
\frac{1}{\sqrt{\eps}}\tilde{V}(\frac{y}{\eps})\Rightarrow \hat{R}(0)^\frac12\dot{\W}(x),
\]
where $\dot{\W}$ is the spatial white noise, and this implies
\[
\frac{1}{\sqrt{\eps}}\int_0^t \tilde{V}(\frac{\eps X_{s/\eps^2}}{\eps})ds\Rightarrow \hat{R}(0)^\frac12\int_0^t\dot{\W}(\bar{\sigma}W_s)ds=\hat{R}(0)^\frac12\int_\R L_t(x)\W(dx),
\]
with $L_t(x)$ the local time of $\bar{\sigma} W_t$ and $\W(dx)$ interpreted as the Wiener integral.

The above heuristic argument has already been made rigorous in \cite[Theorem 3.1]{iftimie2008homogenization}. More precisely, it was shown that for fixed $t>0$, 
\[
(\eps X_{t/\eps^2},  \frac{1}{\sqrt{\eps}}\int_0^t \tilde{V}(\frac{\eps X_{s/\eps^2}}{\eps})ds) \Rightarrow (\bar{\sigma}W_t, \hat{R}(0)^\frac12\int_\R L_t(x)\W(dx)),
\]
where $\bar{\sigma}W_t$ is a Brownian motion built on $(\Sigma,\mathcal{A},\P_B)$, $L_t(x)$ is its local time and $\W(dx)$ is spatial white noise built on $(\Omega,\mathcal{F},\P)$.

In the following, we follow their approach to show the convergence in distribution of $\eps^{-1/2}(v_{1,\eps}+v_{2,\eps})$. To make the argument self-contained, we will provide all details and make appropriate modifications.
 
To simplify the presentation, we will show the weak convergence of $v_{1,\eps}/\sqrt{\eps}$ and $v_{2,\eps}/\sqrt{\eps}$ separately, and in the end it is easy to observe that the proofs combine to show the weak convergence of $\eps^{-1/2}(v_{1,\eps}+v_{2,\eps})$.

\subsection{A decomposition of the probability space}

We decompose $\Omega$ as follows. Define 
\begin{equation}
\W_\eps(x):=\sqrt{\eps}\tilde{\phi}(x/\eps)/\bar{c},
\end{equation}
and since $\W_\eps\Rightarrow \W$ in $\C(\R)$, the family $\{\W_\eps\}$ is tight. For any fixed $\delta>0$, we can find a compact set $K\subset \C(\R)$ such that for all $\eps\in (0,1)$
\[
\P(\W_\eps\in K)>1-\delta.
\]
Clearly $K$ admits an open covering of the form 
\[
\cup_{g\in \C(\R)}\{h\in \C(\R): \sup_{x\in [-\delta^{-1},\delta^{-1}]}|h(x)-g(x)|< \delta\},
\]
so we can extract finitely many deterministic function $g_1,\ldots,g_N\in \C(\R)$ such that 
\[
K\subset \cup_{k=1}^N \{h\in \C(\R): \sup_{x\in [-\delta^{-1},\delta^{-1}]}|h(x)-g_k(x)|< \delta\}.
\]
It is clear that we can further assume $g_k$ is bounded (the bound depends on $\delta$ since $K$ depends on $\delta$). Define 
\begin{equation}
\tilde{B}_k^{\delta,\eps}=\{\omega\in\Omega: \sup_{x\in [-\delta^{-1},\delta^{-1}]}|\W_\eps(x)-g_k(x)|< \delta\},
\label{eq:cover}
\end{equation}
and let $B_1^{\delta,\eps}=\tilde{B}_1^{\delta,\eps}$ and for any $2\leq k\leq N$, 
\[
B_k^{\delta,\eps}=\tilde{B}_k^{\delta,\eps} \setminus \cup_{j=1}^{k-1} B_j^{\delta,\eps},
\]
so we have 
\[
\P(\cup_{k=1}^N \tilde{B}_k^{\delta,\eps})=\sum_{k=1}^N \P(B_k^{\delta,\eps})>1-\delta 
\]
for all $\eps\in (0,1)$. Let $A^{\delta,\eps}=\Omega\setminus \cup_{k=1}^N B_k^{\delta,\eps}$, so $\P(A^{\delta,\eps})<\delta$. Similarly, we define $\tilde{B}_k^{\delta}, B_k^\delta$ and $A^\delta$ with $\W_\eps$ replaced by $\W$ in \eqref{eq:cover}.

The above decomposition of the probability space helps to ``freeze'' the random environment, i.e., with a high probability and a high precision, we can use finitely many deterministic functions to approximate $\W_\eps$. This helps to pass to the limit with only the ``partial'' expectation $\E_B$. 

\subsection{Analysis of $v_{1,\eps}$}

According to the decomposition of the probability space, we have
\[
\begin{aligned}
\frac{v_{1,\eps}}{\sqrt{\eps}}=&- \bar{c}\E_B\{f'(x+M_t^\eps) \W_\eps (\eps X_{t/\eps^2})\}\\
=&-1_{\omega\in A^{\delta,\eps}} \bar{c}\E_B\{f'(x+M_t^\eps) \W_\eps (\eps X_{t/\eps^2})\}-\sum_{k=1}^N 1_{\omega\in B_k^{\delta,\eps}} \bar{c}\E_B\{f'(x+M_t^\eps) \W_\eps (\eps X_{t/\eps^2})\}\\
:=&(i)+(ii).
\end{aligned}
\]

For $(i)$, we first have
\[
|1_{\omega\in A^{\delta,\eps}} \E_B\{f'(x+M_t^\eps) \W_\eps (\eps X_{t/\eps^2})\}|\les 1_{\omega\in A^{\delta,\eps}}\E_B\{|\W_\eps (\eps X_{t/\eps^2})|\}.
\]
By taking $\E$ on both sides and applying Cauchy-Schwartz inequality and Lemma~\ref{lem:resm}, we obtain
\begin{equation}
\begin{aligned}
\E\{|1_{\omega\in A^{\delta,\eps}}\E_B\{f'(x+M_t^\eps) \W_\eps (\eps X_{t/\eps^2})\}|\}\les& \sqrt{\P(A^{\delta,\eps})\E\E_B\{|\W_\eps(\eps X_{t/\eps^2})|^2\}} \\
\les &\sqrt{\delta t^\frac12}.
\end{aligned}
\label{eq:esv1}
\end{equation}

For $(ii)$, we write each summand as 
\[
\begin{aligned}
&1_{\omega\in B_k^{\delta,\eps}}\E_B\{f'(x+M_t^\eps) \W_\eps (\eps X_{t/\eps^2})\}\\
=&1_{\omega\in B_k^{\delta,\eps}}\E_B\{f'(x+M_t^\eps) \W_\eps (\eps X_{t/\eps^2})(1_{|\eps X_{t/\eps^2}|\geq \delta^{-1}}+1_{|\eps X_{t/\eps^2}|< \delta^{-1}})\}.
\end{aligned}
\]
For the first part, we sum over $k$ and write
\[
|\E_B\{f'(x+M_t^\eps) \W_\eps (\eps X_{t/\eps^2})1_{|\eps X_{t/\eps^2}|\geq \delta^{-1}}\}|\les \E_B\{|\W_\eps(\eps X_{t/\eps^2})|1_{|\eps X_{t/\eps^2}|\geq \delta^{-1}}\},
\]
then the same proof as in Lemma~\ref{lem:resm} leads to
\begin{equation}
\E\E_B\{|\W_\eps(\eps X_{t/\eps^2})|1_{|\eps X_{t/\eps^2}|\geq \delta^{-1}}\}\les \int_{|x|\geq \delta^{-1}}\sqrt{|x|}\frac{1}{\sqrt{t}}e^{-c|x|^2/t}dx.
\label{eq:esv2}
\end{equation}
For the other part, we write
\[
\begin{aligned}
&1_{\omega\in B_k^{\delta,\eps}}\E_B\{f'(x+M_t^\eps) \W_\eps (\eps X_{t/\eps^2})1_{|\eps X_{t/\eps^2}|<\delta^{-1}}\}\\
=&1_{\omega\in B_k^{\delta,\eps}}\E_B\{f'(x+M_t^\eps) (\W_\eps-g_k)(\eps X_{t/\eps^2})1_{|\eps X_{t/\eps^2}|<\delta^{-1}}\}\\
&+1_{\omega\in B_k^{\delta,\eps}}\E_B\{f'(x+M_t^\eps) g_k (\eps X_{t/\eps^2})1_{|\eps X_{t/\eps^2}|<\delta^{-1}}\}\\
:=& (iii)+(iv).
\end{aligned}
\]

For $(iii)$, when $\omega\in B_k^{\delta,\eps}$, $|\W_\eps(x)-g_k(x)|<\delta$ for $|x|\leq \delta^{-1}$, so 
\begin{equation}
1_{\omega\in B_k^{\delta,\eps}}|\E_B\{f'(x+M_t^\eps) (\W_\eps-g_k)(\eps X_{t/\eps^2})1_{|\eps X_{t/\eps^2}|<\delta^{-1}}\}|\les \delta 1_{\omega\in B_k^{\delta,\eps}}.
\label{eq:esv3}
\end{equation}

For $(iv)$, we apply the quenched invariance principle of $\eps X_{t/\eps^2}$ (and $M_t^\eps$) to obtain that for almost every $\omega$, 
\begin{equation}
\begin{aligned}
&\E_B\{f'(x+M_t^\eps) g_k (\eps X_{t/\eps^2})1_{|\eps X_{t/\eps^2}|<\delta^{-1}}\}\\
\to& \E_W\{f'(x+\bar{\sigma}W_t)g_k(\bar{\sigma}W_t)1_{|\bar{\sigma}W_t|<\delta^{-1}}\}
\label{eq:qiv1}
\end{aligned}
\end{equation}
as $\eps\to 0$. Furthermore, by the weak convergence of $\W_\eps\Rightarrow \W$, we have
\[
1_{\omega\in B_k^{\delta,\eps}}\Rightarrow 1_{\omega\in B_k^\delta}
\]
in distribution as $\eps\to 0$, since the measure induced by $\W$ on $\C(\R)$ does not support on the boundary of the set $\{h\in \C(\R): \sup_{x\in [-\delta^{-1},\delta^{-1}]}|h(x)-g_k(x)|<\delta\}$. Therefore, we have 
\begin{equation}
\begin{aligned}
&1_{\omega\in B_k^{\delta,\eps}}\E_B\{f'(x+M_t^\eps) g_k (\eps X_{t/\eps^2})1_{|\eps X_{t/\eps^2}|<\delta^{-1}}\}\\\Rightarrow&1_{\omega\in B_k^{\delta}}\E_W\{f'(x+\bar{\sigma}W_t)g_k(\bar{\sigma}W_t)1_{|\bar{\sigma}W_t|<\delta^{-1}}\}
\end{aligned}
\end{equation}
in distribution as $\eps\to 0$. 

To summarize, we can write 
\[
\frac{v_{1,\eps}(t,x)}{\sqrt{\eps}}=-\sum_{k=1}^N1_{\omega\in B_k^{\delta,\eps}}\bar{c}\E_B\{f'(x+M_t^\eps) g_k (\eps X_{t/\eps^2})1_{|\eps X_{t/\eps^2}|<\delta^{-1}}\}+R_{\delta,\eps}
\]
with the first part converges in distribution and $\E\E_B\{|R_{\delta,\eps}|\}
\to 0$ uniformly in $\eps$ as $\delta\to 0$. Now if we write
\[
\begin{aligned}
\E_W\{f'(x+\bar{\sigma}W_t)\W(\bar{\sigma}W_t)\}=&1_{\omega\in A^\delta}\E_W\{f'(x+\bar{\sigma}W_t)\W(\bar{\sigma}W_t)\}\\
&+\sum_{k=1}^N 1_{\omega\in B_k^\delta}\E_W\{f'(x+\bar{\sigma}W_t)\W(\bar{\sigma}W_t)1_{|\bar{\sigma}W_t|\geq \delta^{-1}}\}\\
&+\sum_{k=1}^N 1_{\omega\in B_k^\delta}\E_W\{f'(x+\bar{\sigma}W_t)(\W-g_k)(\bar{\sigma}W_t)1_{|\bar{\sigma}W_t|< \delta^{-1}}\}\\
&+\sum_{k=1}^N 1_{\omega\in B_k^\delta}\E_W\{f'(x+\bar{\sigma}W_t)g_k(\bar{\sigma}W_t)1_{|\bar{\sigma}W_t|< \delta^{-1}}\},
\end{aligned}
\]
by the same discussion, we have similar estimates for the first three terms in the above expression as in \eqref{eq:esv1}, \eqref{eq:esv2} and \eqref{eq:esv3}. Now we only need to send $\delta\to 0$ to conclude that 
\begin{equation}
\frac{v_{1,\eps}(t,x)}{\sqrt{\eps}}\Rightarrow -\bar{c}\E_W\{f'(x+\bar{\sigma}W_t)\W(\bar{\sigma}W_t)\}
\label{eq:conv1}
\end{equation}
in distribution as $\eps\to0$.

\subsection{Analysis of $v_{2,\eps}$}

By the proof of Lemma~\ref{lem:masm}, we have
\[
\begin{aligned}
\frac{\langle M^\eps\rangle_t-\bar{a}t}{\bar{a}^2\sqrt{\eps}}=&\eps^{\frac32}\int_0^{t/\eps^2}\tilde{V}(X_s)ds\\
=&-\eps^\frac32\tilde{\psi}(X_{t/\eps^2})+\eps^\frac32\int_0^{t/\eps^2}\tilde{\psi}'(X_s)\tilde{\sigma}(X_s)dB_s,
\end{aligned}
\]
with $\tilde{\psi}$ given by
\[
\tilde{\psi}(x)=-\frac{2}{\bar{a}}\int_0^x \frac{\tilde{\phi}(y)}{\tilde{a}(y)}dy.
\]
Since $\W_\eps(x)=\sqrt{\eps}\tilde{\phi}(x/\eps)/\bar{c}$, we can write
\begin{equation}
\frac{\langle M^\eps\rangle_t-\bar{a}t}{\bar{a}^2\sqrt{\eps}}=\frac{2\bar{c}}{\bar{a}}\left(\int_0^{\eps X_{t/\eps^2}} \frac{\W_\eps(y)}{\tilde{a}(y/\eps)}dy-\int_0^t\frac{\W_\eps(\eps X_{s/\eps^2})}{\tilde{\sigma}(\eps X_{s/\eps^2}/\eps)}d\tilde{B}_s\right),
\label{eq:maqua}
\end{equation}
where $\tilde{B}_s:=\eps B_{s/\eps^2}$. 

The idea is the same as for $v_{1,\eps}$, i.e., we decompose the probability $\Omega$ as 
\[
\Omega=(\cup_{k=1}^N B_k^{\delta,\eps})\cup A^{\delta,\eps}
\]
and for $\omega\in B_k^{\delta,\eps}$, we use the deterministic function $g_k$ to approximate $\W_\eps$ and pass to the limit by the invariance principle of $\eps X_{t/\eps^2}$.

First, for $\omega\in A^{\delta,\eps}$, by Cauchy-Schwartz inequality and Lemma~\ref{lem:masm}, we have
\begin{equation}
\E\{|1_{\omega\in A^{\delta,\eps}}\E_B\{f''(x+M_t^\eps)(\langle M^\eps\rangle_t-\bar{a}t)/\sqrt{\eps}\}|\}\les \sqrt{\delta t^\frac32}.
\label{eq:esv4}
\end{equation}

Secondly, we apply Cauchy-Schwartz inequality and \eqref{eq:heatkp} to derive that
\[
\begin{aligned}
&|\E_B\{1_{\sup_{s\in [0,t]}|\eps X_{s/\eps^2}|\geq \delta^{-1}}f''(x+M_t^\eps)(\langle M^\eps\rangle_t-\bar{a} t)/\sqrt{\eps} \}|^2\\
\les & \E_B\{1_{\sup_{s\in [0,t]}|\eps X_{s/\eps^2}|\geq \delta^{-1}}\}\E_B\{ (\langle M^\eps\rangle_t-\bar{a} t)^2/\eps\}\\
\les &  e^{-c\delta^{-2}/t}\E_B\{ (\langle M^\eps\rangle_t-\bar{a} t)^2/\eps\}.
\end{aligned}
\]
By taking $\E$ on both sides and applying Lemma~\ref{lem:masm}, we obtain
\begin{equation}
|\E\E_B\{1_{\sup_{s\in [0,t]}|\eps X_{s/\eps^2}|\geq \delta^{-1}}f''(x+M_t^\eps)(\langle M^\eps\rangle_t-\bar{a} t)/\sqrt{\eps} \}|\les t^{\frac34}e^{-c\delta^{-2}/t}.
\label{eq:esv5}
\end{equation}


Now we define
\[
\cG(\W_\eps):=\frac{\langle M^\eps\rangle_t-\bar{a}t}{\sqrt{\eps}}=2\bar{c}\bar{a}\left(\int_0^{\eps X_{t/\eps^2}} \frac{\W_\eps(y)}{\tilde{a}(y/\eps)}dy-\int_0^t\frac{\W_\eps(\eps X_{s/\eps^2})}{\tilde{\sigma}(\eps X_{s/\eps^2}/\eps)}d\tilde{B}_s\right),
\]
and consider the error induced by replacing $\W_\eps\to g_k$ when $\omega\in B_k^{\delta,\eps}$. Since 
\[
|\W_\eps(y)-g_k(y)|\leq \delta
\]
for $y\in [-\delta^{-1},\delta^{-1}]$ and $\omega\in B_k^{\delta,\eps}$, we have
\begin{equation}
\begin{aligned}
&1_{\omega\in B_k^{\delta,\eps}} |\E_B\{1_{\sup_{s\in [0,t]} |\eps X_{s/\eps^2}|<\delta^{-1}}f''(x+M_t^\eps)(\cG(\W_\eps)-\cG(g_k))\}|\\
\les &1_{\omega\in B_k^{\delta,\eps}} \left(\E_B\{ |\eps X_{t/\eps^2}|\delta\}+\sqrt{\E_B\{\int_0^t (\W_\eps-g_k)^2(\eps X_{s/\eps^2})1_{|\eps X_{s/\eps^2}|< \delta^{-1}}ds\}}\right)\\
\les & 1_{\omega\in B_k^{\delta,\eps}}\left(\E_B\{|\eps X_{t/\eps^2}|\delta\}+\delta \sqrt{t}\right)\les 1_{\omega\in B_k^{\delta,\eps}} \delta\sqrt{t}.
\end{aligned}
\label{eq:esv6}
\end{equation}
Here we have used the simple fact that
\[
\begin{aligned}
&1_{\sup_{s\in [0,t]}|\eps X_{s/\eps^2}|<\delta^{-1}}|\int_0^t \frac{(\W_\eps-g_k)(\eps X_{s/\eps^2})}{\tilde{\sigma}(\eps X_{s/\eps^2}/\eps)}d\tilde{B}_s|\\
\leq&|\int_0^{t\wedge \tau_\delta}\frac{(\W_\eps-g_k)(\eps X_{s/\eps^2})}{\tilde{\sigma}(\eps X_{s/\eps^2}/\eps)}d\tilde{B}_s|
\end{aligned}
\]
and It\^o's isometry, where $\tau_\delta:=\inf\{ s\geq 0: \eps X_{s/\eps^2}\geq \delta^{-1}\}$.


By combining \eqref{eq:esv4}, \eqref{eq:esv5}, and \eqref{eq:esv6}, it is clear that
\begin{equation}
\frac{v_{2,\eps}(t,x)}{\sqrt{\eps}}=\frac12\sum_{k=1}^N 1_{\omega\in B_k^{\delta,\eps}}\E_B\{1_{\sup_{s\in [0,t]}|\eps X_{s/\eps^2}|<\delta^{-1}}f''(x+M_t^\eps)\cG(g_k)\}+R_{\delta,\eps}
\label{eq:edev2}
\end{equation}
with $R_{\delta,\eps}\to 0$ in $L^1(\Omega)$ as $\delta\to 0$ uniformly in $\eps$. 

Now we consider
\[
1_{\omega\in B_k^{\delta,\eps}}\E_B\{1_{\sup_{s\in [0,t]}|\eps X_{s/\eps^2}|<\delta^{-1}}f''(x+M_t^\eps)\cG(g_k)\}
\]
for fixed $k$. As before, for $1_{\omega\in B_k^{\delta,\eps}}\Rightarrow 1_{\omega\in B_k^\delta}$ in distribution by the weak convergence of $\W_\eps\Rightarrow \W$, so we only need to prove the convergence of 
\[
\E_B\{1_{\sup_{s\in [0,t]}|\eps X_{s/\eps^2}|<\delta^{-1}}f''(x+M_t^\eps)\cG(g_k)\}.
\]

Let 
\[
\begin{aligned}
\cG(g_k)=&
2\bar{c}\bar{a}\left(\int_0^{\eps X_{t/\eps^2}} \frac{g_k(y)}{\tilde{a}(y/\eps)}dy-\int_0^t\frac{g_k(\eps X_{s/\eps^2})}{\tilde{\sigma}(\eps X_{s/\eps^2}/\eps)}d\tilde{B}_s\right)\\
:=&2\bar{c}\bar{a}(\mathcal{R}^\eps_t-\mathcal{M}^\eps_t).
\end{aligned}
\]

For $\mathcal{R}_t^\eps$, by ergodic theorem, we have
\[
1_{\sup_{s\in [0,t]}|\eps X_{s/\eps^2}|<\delta^{-1}}\int_0^{\eps X_{t/\eps^2}}g_k(y)\left(\frac{1}{\tilde{a}(y/\eps)}-\frac{1}{\bar{a}}\right)dy\to 0
\]
as $\eps\to0$ for almost every $\omega$ and $\tilde{B}_s$. Thus by the quenched convergence of $\eps X_{t/\eps^2}\Rightarrow \bar{\sigma}W_t$ in $\C(\R_+)$, we have
\begin{equation}
\begin{aligned}
&\E_B\{ 1_{\sup_{s\in [0,t]} |\eps X_{s/\eps^2}|<\delta^{-1}}f''(x+M_t^\eps) \mathcal{R}_t^\eps\}\\
\to &\E_W\{1_{\sup_{s\in [0,t]} |\bar{\sigma}W_s|<\delta^{-1}}f''(x+\bar{\sigma}W_t) \frac{1}{\bar{a}}\int_0^{\bar{\sigma}W_t}g_k(y)dy\}
\end{aligned}
\label{eq:conr1}
\end{equation}
for almost every $\omega$.

For the continuous martingale 
\[
\mathcal{M}^\eps_t=\int_0^t \frac{g_k(\eps X_{s/\eps^2})}{\tilde{\sigma}(\eps X_{s/\eps^2}/\eps)}d\tilde{B}_s,
\]
its quadratic variation is given by
\[
\langle \mathcal{M}^\eps\rangle_t=\int_0^t \frac{g_k^2(\eps X_{s/\eps^2})}{\tilde{a}(\eps X_{s/\eps^2}/\eps)}ds,
\]
and recall that in the proof of Proposition~\ref{p:qiv}, $\eps X_{t/\eps^2}$ is decomposed as $\eps X_{t/\eps^2}=R_t^\eps+M_t^\eps$
with 
\[
M_t^\eps=\bar{a}\int_0^t \frac{1}{\tilde{\sigma}(\eps X_{s/\eps^2}/\eps)}d\tilde{B}_s,
\]
 so the joint quadratic variation of $\mathcal{M}^\eps_t$ and $M_t^\eps$ is 
\[
\langle \mathcal{M}^\eps,M^\eps\rangle_t=\bar{a}\int_0^t \frac{g_k(\eps X_{s/\eps^2})}{\tilde{a}(\eps X_{s/\eps^2}/\eps)}ds.
\]
The following lemma shows the joint convergence of $\M^\eps_t$ and $M_t^\eps$.

\begin{lem}
For almost every $\omega$, 
\[
(\M^\eps_t,M_t^\eps)\Rightarrow (\frac{1}{\bar{\sigma}}\int_0^t g_k(\bar{\sigma}W_s)dW_s, \bar{\sigma} W_t)
\]
in $\C(\R_+)$ as $\eps\to 0$.
\label{lem:conma}
\end{lem}

\begin{proof}
We first consider the process on $\C(\R_+)$ defined by
\[
h_\eps(t):=\int_0^t \frac{1}{\tilde{a}(\eps X_{s/\eps^2}/\eps)}ds=\eps^2\int_0^{t/\eps^2}\frac{1}{a(\tau_{X_s}\omega)}ds.
\]
Since $a^{-1}$ is bounded, $h_\eps(.)$ is tight in $\C(\R_+)$, and by ergodic theorem, its finite dimensional distribution converges to that of $\bar{h}(t):=t/\bar{a}$. Therefore, for almost every $\omega$, we have $h_\eps\Rightarrow h$ in $\C(\R_+)$.

Secondly, since $\eps X_{t/\eps^2}\Rightarrow \sigma W_t$ in $\C(\R_+)$, we have $g_k^2(\eps X_{s/\eps^2})\Rightarrow g_k^2(\bar{\sigma} W_s)$ in $\C(\R_+)$. Now by \cite[Lemma 3.5]{iftimie2008homogenization}, the following mapping is continuous from $\C(\R_+)\times \C(\R_+)\to \C(\R_+)$:
\[
(h_1,h_2)\mapsto \int_0^. h_1(s)dh_2(s),
\]
when $h_2$ is increasing and $\C(\R_+)$ is equipped with the locally uniform topology. This implies 
\[
\langle \M^\eps\rangle_t=\int_0^t \frac{g_k^2(\eps X_{s/\eps^2})}{\tilde{a}(\eps X_{s/\eps^2}/\eps)}ds\Rightarrow \frac{1}{\bar{a}}\int_0^tg_k^2(\bar{\sigma}W_s)ds
\]
in $\C(\R_+)$. Similarly, 
\[
\langle \M^\eps,M^\eps\rangle_t=\bar{a}\int_0^t \frac{g_k(\eps X_{s/\eps^2})}{\tilde{a}(\eps X_{s/\eps^2}/\eps)}ds\Rightarrow \int_0^t g_k(\bar{\sigma}W_s)ds
\]
in $\C(\R_+)$.

Now given the fact that 
\[
\langle M^\eps\rangle_t\Rightarrow \bar{a}t,
\]
we conclude by martingale central limit theorem that 
\[
(\M^\eps_t,M^\eps_t)\Rightarrow (\frac{1}{\bar{\sigma}}\int_0^t g_k(\bar{\sigma}W_s)dW_s, \bar{\sigma} W_t)
\]
in $\C(\R_+)$ for almost every $\omega$. The proof is complete.
\end{proof}

From Lemma~\ref{lem:conma} it is clear that 
\[
\begin{aligned}
&1_{\sup_{s\in[0,t]}|\eps X_{s/\eps^2}|<\delta^{-1}} f''(x+M_t^\eps) \M^\eps_t\\
\Rightarrow & 1_{\sup_{s\in[0,t]}|\bar{\sigma}W_t|<\delta^{-1}} f''(x+\bar{\sigma} W_t) \frac{1}{\bar{\sigma}}\int_0^t g_k(\bar{\sigma}W_s)dW_s
\end{aligned}
\]
in distribution for almost every $\omega$. Now we only need to note
\[
1_{\sup_{s\in[0,t]}|\eps X_{s/\eps^2}|<\delta^{-1}} | \M^\eps_t|\leq |\int_0^{t\wedge \tau_\delta}\frac{g_k(\eps X_{s/\eps^2})}{\tilde{\sigma}(\eps X_{s/\eps^2}/\eps)}d\tilde{B}_s|,
\]
which implies $\E_B\{1_{\sup_{s\in[0,t]}|\eps X_{s/\eps^2}|<\delta^{-1}}|\M_t^\eps|^2\}$ is uniformly bounded in $\eps$, so
\begin{equation}
\begin{aligned}
&\E_B\{ 1_{\sup_{s\in[0,t]}|\eps X_{s/\eps^2}|<\delta^{-1}} f''(x+M_t^\eps) \M^\eps_t\}\\
\to & \E_W\{1_{\sup_{s\in[0,t]}|\bar{\sigma} W_s|<\delta^{-1}} f''(x+\bar{\sigma}W_t)\frac{1}{\bar{\sigma}}\int_0^t g_k(\bar{\sigma}W_s)dW_s\}
\end{aligned}
\label{eq:conm1}
\end{equation}
for almost every $\omega\in \Omega$.

Combining \eqref{eq:conr1} and \eqref{eq:conm1}, we obtain
\begin{equation}
\begin{aligned}
&\sum_{k=1}^N 1_{\omega\in B_k^{\delta,\eps}}\E_B\{ 1_{\sup_{s\in [0,t]}|\eps X_{s/\eps^2}|<\delta^{-1}}f''(x+M_t^\eps) \cG(g_k)\}\\
\Rightarrow &\sum_{k=1}^N 1_{\omega\in B_k^\delta}\E_W\{1_{\sup_{s\in [0,t]}|\bar{\sigma}W_s|<\delta^{-1}}f''(x+\bar{\sigma}W_t)\left(2\bar{c}\int_0^{\bar{\sigma}W_t}g_k(y)dy-2\bar{c}\bar{\sigma}\int_0^t g_k(\bar{\sigma}W_s)dW_s\right)\}
\end{aligned}
\label{eq:conv2cut}
\end{equation}
in distribution as $\eps\to 0$.

Now we send $\delta\to 0$ on the r.h.s. of \eqref{eq:conv2cut} to obtain 
\[
\mbox{r.h.s. of } \eqref{eq:conv2cut} \to \E_W\{f''(x+\bar{\sigma}W_t)\left(2\bar{c}\int_0^{\bar{\sigma}W_t}\W(y)dy-2\bar{c}\bar{\sigma}\int_0^t \W(\bar{\sigma}W_s)dW_s\right)\}
\]
in $L^1(\Omega)$. The discussion is the same as for \eqref{eq:esv4}, \eqref{eq:esv5} and \eqref{eq:esv6}.

To summarize, we have 
\begin{equation}
\frac{v_{2,\eps}(t,x)}{\sqrt{\eps}}\Rightarrow \E_W\{f''(x+\bar{\sigma}W_t)\left(\bar{c}\int_0^{\bar{\sigma}W_t}\W(y)dy-\bar{c}\bar{\sigma}\int_0^t \W(\bar{\sigma}W_s)dW_s\right)\}.
\label{eq:conv2}
\end{equation}

Now we note that the proofs of \eqref{eq:conv1} and \eqref{eq:conv2} can actually be combined together to show that
\begin{equation}
\begin{aligned}
\frac{v_{1,\eps}(t,x)+v_{2,\eps}(t,x)}{\sqrt{\eps}} \Rightarrow&-\bar{c}\E_W\{f'(x+\bar{\sigma}W_t)\W(\bar{\sigma}W_t)\}\\
&+\E_W\{f''(x+\bar{\sigma}W_t)\left(\bar{c}\int_0^{\bar{\sigma}W_t}\W(y)dy-\bar{c}\bar{\sigma}\int_0^t \W(\bar{\sigma}W_s)dW_s\right)\}.
\end{aligned}
\label{eq:conv1v2}
\end{equation}

By \cite[Lemma 3.12]{iftimie2008homogenization}, we have
\begin{equation}
\bar{c}\int_0^{\bar{\sigma}W_t}\W(y)dy-\bar{c}\bar{\sigma}\int_0^t \W(\bar{\sigma}W_s)dW_s=\frac12\hat{R}(0)^\frac12\bar{a}^2\int_\R L_t(x)\W(dx),
\label{eq:ito}
\end{equation}
with $L_t(x)$ the local time of $\bar{\sigma}W_t$. For simplicity we formally write the r.h.s. of the above expression as
\[
\frac12\hat{R}(0)^\frac12\bar{a}^2\int_\R L_t(x)\W(dx)=\frac12\hat{R}(0)^\frac12\bar{a}^2\int_0^t \dot{\W}(\bar{\sigma} W_s)ds,
\]
so \eqref{eq:ito} can be viewed as an application of It\^o's formula.

Now \eqref{eq:conv1v2} is rewritten as (recall that $\bar{c}=\hat{R}(0)^\frac12\bar{a}$)
\begin{equation}
\begin{aligned}
\frac{v_{1,\eps}(t,x)+v_{2,\eps}(t,x)}{\sqrt{\eps}} \Rightarrow&-\hat{R}(0)^\frac12\bar{a}\E_W\{f'(x+\bar{\sigma}W_t)\W(\bar{\sigma}W_t)\}\\
&+\frac12\hat{R}(0)^\frac12\bar{a}^2\E_W\{f''(x+\bar{\sigma}W_t)\int_0^t \dot{\W}(\bar{\sigma} W_s)ds\}.
\label{eq:connew}
\end{aligned}
\end{equation}

\section{Discussion}
\label{s:dis}

\subsection{The shift of the random environment and global weak convergence}

The weak convergence in \eqref{eq:connew} is for fixed $(t,x)$. Now we discuss the convergence of finite dimensional distributions of $v(t,x)$ as a process in $(t,x)$. 

From the beginning, we have shifted the random environment $\omega$ by $\tau_{-x/\eps}$, and this is only used when applying martingale central limit theorem to obtain \emph{quenched} weak convergence. The reason is that we need the environmental process $\omega_s$ to be independent of $\eps$, and the shift of the environment enables the process $\omega_s=\tau_{X_s}\omega$ to start from $\omega$ instead of $\tau_{x/\eps}\omega$.

Retracing the proof, the shift of the environment is used in proving Proposition~\ref{p:qiv}, \eqref{eq:qiv1}, \eqref{eq:conr1}, Lemma~\ref{lem:conma} and \eqref{eq:conm1} to get almost sure convergence, which is sufficient but unnecessary. For example, in \eqref{eq:qiv1}, \eqref{eq:conr1} and \eqref{eq:conm1}, a convergence in probability suffices to pass to the limit, which itself could come from an $L^1(\Omega)$ convergence, since the $L^1(\Omega)$ error does not depend on where the environmental process starts. Therefore, by almost the same proof, we have a convergence of finite dimensional distributions, except that now the limit in \eqref{eq:connew} should encode the dependence on $x$.

Without the shift, by the same proof \eqref{eq:connew} becomes
\begin{equation}
\frac{v_{1,\eps}(t,x)+v_{2,\eps}(t,x)}{\sqrt{\eps}}\Rightarrow v(t,x),
\end{equation}
with 
\begin{equation}
\begin{aligned}
v(t,x)=&-\hat{R}(0)^\frac12\bar{a}\E_W\{f'(x+\bar{\sigma}W_t)(\W(x+\bar{\sigma}W_t)-\W(x))\}\\
&+\frac12\hat{R}(0)^\frac12\bar{a}^2\E_W\{f''(x+\bar{\sigma}W_t)\int_0^t \dot{\W}(x+\bar{\sigma} W_s)ds\}.
\end{aligned}
\label{eq:conv1v2n}
\end{equation}
Comparing with \eqref{eq:conv1v2n} and \eqref{eq:connew}, formally it is only to shift $\dot{\W}(y)\mapsto \dot{\W}(x+y)$, i.e.,
\[
\W(\bar{\sigma} W_t)=\int_0^{\bar{\sigma}W_t}\dot{\W}(y)dy\mapsto \int_0^{\bar{\sigma}W_t}\dot{\W}(x+y)dy=\W(x+\bar{\sigma}W_t)-\W(x),
\]
and
\[
\int_0^t\dot{\W}(\bar{\sigma}W_s)ds\mapsto \int_0^t \dot{\W}(x+\bar{\sigma}W_s)ds.
\]
Since the spatial white noise $\dot{\W}$ is stationary, the pointwise distribution is unchanged.

Now we prove the ``global'' weak convergence, i.e., a spatial average of the homogenization error $u_\eps-u_{\hom}$ of the form 
\[
\frac{1}{\sqrt{\eps}}\int_\R (u_\eps(t,x)-u_{\hom}(t,x))g(x)dx
\]
 converges for every test function $g\in \C_c^\infty(\R)$.

First, we point out that the estimates derived in the proofs of Lemma~\ref{lem:resm} and \ref{lem:masm} grows polynomially with respect to $|x|$ when we have $X_0=x/\eps$ instead of $X_0=0$. Since $g$ is fast-decaying, we have
\[
\int_\R \E\{ |(u_\eps(t,x)-u_{\hom}(t,x)-v_{1,\eps}(t,x)-v_{2,\eps}(t,x))g(x)|dx\}\ll \sqrt{\eps}.
\]
It still  remains to analyze 
\[
\frac{1}{\sqrt{\eps}}\int_\R (v_{1,\eps}(t,x)+v_{2,\eps}(t,x))g(x)dx.
\]
The previous argument goes through except for some modification in the proof of \eqref{eq:qiv1}, \eqref{eq:conr1} and \eqref{eq:conm1}. Taking \eqref{eq:qiv1} for example,  it suffices to show that
\begin{equation}
\begin{aligned}
&\int_\R g(x) \E_B\{ f'(x+M_t^\eps)g_k(\eps X_{t/\eps^2})1_{|\eps X_{t/\eps^2}|<\delta^{-1}}\}dx\\
\to&\int_\R g(x) \E_W\{ f'(x+\bar{\sigma}W_t) g_k(x+\bar{\sigma}W_t)1_{|x+\bar{\sigma}W_t|<\delta^{-1}}\}dx
\label{eq:qiv1n}
\end{aligned}
\end{equation}
in probability as $\eps\to 0$. We emphasize that $\eps X_{t/\eps^2}=x+R_t^\eps+M_t^\eps$ and the environmental process $\omega_s=\tau_{X_s}\omega$ starts from $\omega_0=\tau_{x/\eps}\omega$. To prove \eqref{eq:qiv1n}, we consider the $L^1(\Omega)$ error
\begin{equation}
\int_\R |g(x)| \E\{|I_1(x)-I_2(x)|\}dx,
\label{eq:qiv1nn}
\end{equation}
with 
\begin{eqnarray*}
I_1(x)&=& \E_B\{ f'(x+M_t^\eps)g_k(\eps X_{t/\eps^2})1_{|\eps X_{t/\eps^2}|<\delta^{-1}}\},\\
I_2(x)&=&\E_W\{ f'(x+\bar{\sigma}W_t) g_k(x+\bar{\sigma}W_t)1_{|x+\bar{\sigma}W_t|<\delta^{-1}}\}.
\end{eqnarray*}
For every fixed $x\in \R$, we can again shift the environment by $\tau_{-x/\eps}$ without affecting the value of $\E\{|I_1(x)-I_2(x)|\}$, and after the change of the environment, by the quenched invariance principle, $\E\{|I_1(x)-I_2(x)|\}\to 0$  as $\eps \to0$. Since it is uniformly bounded, we only need to apply dominated convergence theorem to derive $\eqref{eq:qiv1nn}\to 0$, which further implies \eqref{eq:qiv1n}.

To summarize, we have shown that 
\[
v_\eps(t,x)=\frac{u_\eps(t,x)-u_{\hom}(t,x)}{\sqrt{\eps}}\Rightarrow v(t,x)
\]
in the sense of Theorem~\ref{t:mainth}, with $v(t,x)$ given by \eqref{eq:conv1v2n}.


\subsection{PDE representations and a comparison with high dimensions}

Now we discuss the individual terms coming from $v(t,x)$. By \eqref{eq:conv1v2n}, let us write 
\[
v(t,x)=-\hat{R}(0)^\frac12\bar{a}v_1(t,x)+\hat{R}(0)^\frac12\bar{a}v_2(t,x)+\frac12\hat{R}(0)^\frac12\bar{a}^2v_3(t,x),
\]
 with
\begin{eqnarray*}
v_1(t,x)&=&\E_W\{f'(x+\bar{\sigma}W_t)\W(x+\sigma W_t)\},\\
v_2(t,x)&=&\E_W\{f'(x+\bar{\sigma}W_t)\}\W(x),\\
v_3(t,x)&=&\E_W\{f''(x+\bar{\sigma}W_t)\int_0^t \dot{\W}(x+\bar{\sigma} W_s)ds\}.
\end{eqnarray*}
It is easy to see that $v_1(t,x)$ is the solution to the heat equation with a random initial condition, i.e., 
\begin{equation}
\partial_t v_1=\frac12\bar{a}\partial_{x}^2v_1, \mbox{ with } v_1(0,x)=f'(x)\W(x),
\label{eq:v1}
\end{equation}
and
 \begin{equation}
 v_2(t,x)=\partial_x u_{\hom}(t,x)\W(x).
 \label{eq:v2}
 \end{equation}
 For $v_3$, Lemma~\ref{lem:spde} shows that it solves the SPDE with additive spatial white noise and zero initial condition:
 \begin{equation}
 \partial_t v_3=\frac12\bar{a}\partial_{x}^2v_3+\partial_{x}^2u_{\hom}(t,x)\dot{\W}(x), \mbox{ with } v_3(0,x)=0.
 \label{eq:v3}
 \end{equation}
%

We point out that $v_2$ corresponds to the first order fluctuation obtained by a formal two scale expansion. If we write
\[
u_\eps(t,x)=u_{\hom}(t,x)+\eps u_1(t,x,\frac{x}{\eps})+\ldots,
\]
then $u_1(t,x,x/\eps)=\partial_x u_{\hom}(t,x)\tilde{\phi}(x/\eps)$ with $\tilde{\phi}$ solving the corrector equation \eqref{eq:coreq}. Since $\sqrt{\eps}\tilde{\phi}(x/\eps)$ scales to a two-sided Brownian motion (which is not stationary) when $d=1$:
\[
\sqrt{\eps}\tilde{\phi}(\frac{x}{\eps}) \Rightarrow \hat{R}(0)^\frac12\bar{a}\W(x),
\]
we have
\[
\eps \partial_x u_{\hom}(t,x)\tilde{\phi}(\frac{x}{\eps})\sim \sqrt{\eps}\partial_x u_{\hom}(t,x)\hat{R}(0)^\frac12\bar{a}\W(x)=\sqrt{\eps}\hat{R}(0)^\frac12\bar{a}v_2(t,x).
\]

The first order fluctuation given by $v(t,x)$ is very different in high dimensions $d\geq 3$. Recall that $u_\eps-u_{\hom}\approx v_{1,\eps}+v_{2,\eps}$ with $v_{1,\eps}(t,x)=\E_B\{f'(x+M_t^\eps)R_t^\eps\}$ and 
\[
R_t^\eps=-\eps \tilde{\phi}(X_{t/\eps^2})+\eps \tilde{\phi}(X_0).
\]
When $d\geq 3$, we have a stationary zero-mean corrector \cite[Corollary 1]{gloria2014quantitative}, so 
\[
\E_B\{ f'(x+M_t^\eps) \eps \tilde{\phi}(X_0)\}\sim \eps \partial_x u_{\hom}(t,x)\tilde{\phi}(\frac{x}{\eps}),
\]
and 
\[
|\E_B\{ f'(x+M_t^\eps) \eps \tilde{\phi}(X_{t/\eps^2})\}|\ll \eps
\]
due to the fact that $\E\{\tilde{\phi}\}=0$ and the mixing induced by $X_{t/\eps^2}$ when $\eps$ is small. For $v_{2,\eps}(t,x)=\frac12\E_B\{f''(x+M_t^\eps)(\langle M^\eps\rangle_t-\bar{a}t)\}$, it turns out  
\[
\eps^{-1}(\langle M^\eps\rangle_t-\bar{a}t) =\bar{a}^2\eps\int_0^{t/\eps^2}\tilde{V}(X_s)ds
\]
is an approximating martingale when $d\geq 3$, and is asymptotically independent of $M_t^\eps$. This implies $|v_{2,\eps}(t,x)|\ll \eps$. Combining these results, it was shown for fixed $(t,x)$ that 
\[
u_\eps(t,x)=u_{\hom}(t,x)+\eps \nabla u_{\hom}(t,x)\cdot \tilde{\phi}(\frac{x}{\eps})+o(\eps),
\]
with $o(\eps)/\eps\to 0$ in $L^1(\Omega)$. Therefore, the pointwise first order fluctuation when $d\geq 3$ is given by $\eps \nabla u_{\hom}(t,x)\cdot \tilde{\phi}(x/\eps)$, which only corresponds to $v_2(t,x)=\partial_x u_{\hom}(t,x)\W(x)$ when $d=1$.

The following simple example illustrates the differences. Let $u_\eps(0,x)=\xi\cdot x$ for some fixed direction $\xi\in \R^d$, so 
\[
u_\eps(t,x)=\E_B\{\xi\cdot \eps X_{t/\eps^2}\}=\xi\cdot x-\eps \E_B\{\xi\cdot \tilde{\phi}(X_{t/\eps^2})\}+\eps \xi\cdot \tilde{\phi}(X_0)
\]
When a stationary corrector exists in $d\geq 3$, $\E_B\{\xi\cdot \tilde{\phi}(X_{t/\eps^2})\}\to 0$ in $L^2(\Omega)$ as $\eps\to 0$, and this is not the case by our proof when $d=1$.

To summarize, the underlying diffusion process is so recurrent when $d=1$ that the sample path is recorded in the asymptotic limit as $\eps\to 0$, and all three terms in $v_{1,\eps}+v_{2,\eps}$ contribute to the first order fluctuation. When $d\geq 3$, we have sufficient mixing effects coming from the diffusion process, which leads to a different asymptotic behavior.

\subsection{An SPDE representation of $v(t,x)$}

At this point, our proof shows the limit $v(t,x)$ is a superposition of three Gaussian processes $v_1,v_2,v_3$, and it turns out that they can be combined to form the solution to the SPDE given by \eqref{eq:limitspde}:
\begin{prop}
$v(t,x)$ solves 
\begin{equation}
\partial_t v=\frac12\bar{a}\partial_x^2 v-\frac12\hat{R}(0)^\frac12\bar{a}^2 \partial_x (\partial_x u_{\hom}\dot{\W}), \mbox{ with } v(0,x)=0.
\label{eq:limitspde}
\end{equation}
\label{p:spde}
\end{prop}

We first give a heuristic derivation of \eqref{eq:limitspde}. Recall that 
\[
v_3(t,x)=\E_W\{f''(x+\bar{\sigma}W_t)\int_0^t \dot{\W}(x+\bar{\sigma}W_s)ds\},
\]
and if we treat $\dot{\W}$ as a function, an application of duality relation in Malliavin calculus shows that
\[
v_3(t,x)=\frac{1}{\bar{\sigma}}\E_W\{f'(x+\bar{\sigma}W_t)\int_0^t \dot{\W}(x+\bar{\sigma}W_s)dW_s\}.
\]
Furthermore, since $v_1(t,x)-v_2(t,x)=\E_W\{f'(x+\bar{\sigma}W_t)(\W(x+\bar{\sigma}W_t)-\W(x))$, a formal application of It\^o's formula gives that
\[
v_1(t,x)-v_2(t,x)-\bar{\sigma}^2v_3(t,x)=\frac12\bar{\sigma}^2\E_W\{f'(x+\bar{\sigma}W_t)\int_0^t \ddot{\W}(x+\bar{\sigma}W_s)ds\},
\]
so by recalling that $\bar{a}=\bar{\sigma}^2$, we obtain
\[
\begin{aligned}
v(t,x)=&-\hat{R}(0)^\frac12\bar{a}(v_1(t,x)-v_2(t,x)-\frac12\bar{a}v_3(t,x))\\
=&-\frac12\hat{R}(0)^\frac12\bar{a}^2 (v_3+v_4),
\end{aligned}
\]
with 
\[
v_4(t,x)=\E_W\{f'(x+\bar{\sigma}W_t)\int_0^t \ddot{\W}(x+\bar{\sigma}W_s)ds\}.
\]
Since $v_3$ solves $\partial_tv_3=\frac12\bar{a}\partial_x^2v_3+\partial_x^2u_{\hom}\dot{\W}$ with zero initial data, the same argument should predict $v_4$ solves
\[
\partial_t v_4=\frac12\bar{a}\partial_x^2v_4+\partial_xu_{\hom}\ddot{\W}, \mbox{ with } v_4(0,x)=0, 
\]
hence $v$ should satisfy
\[
\partial_t v=\frac12\bar{a}\partial_x^2v-\frac12\hat{R}(0)^\frac12\bar{a}^2(\partial_{x}^2u_{\hom}\dot{\W}+\partial_xu_{\hom}\ddot{\W}), \mbox{ with } v(0,x)=0,
\]
which leads to \eqref{eq:limitspde} if we write $\partial_x(\partial_x u_{\hom}\dot{\W})=\partial_{x}^2u_{\hom}\dot{\W}+\partial_xu_{\hom}\ddot{\W}$. 

The following is a rigorous proof of the above argument by introducing some mollification.

\begin{proof}[Proof of Proposition~\ref{p:spde}]
For fixed $(t,x)$, the solution to \eqref{eq:limitspde} can be written as
\[
\begin{aligned}
v(t,x)=&-\frac12\hat{R}(0)^\frac12\bar{a}^2\int_\R\left(\int_0^t  q_{\bar{a}(t-s)}'(x-y)\partial_yu_{\hom}(s,y)ds\right)\W(dy)\\
:=&\int_\R G(y)\W(dy).
\end{aligned}
\]
It is straightforward to check that $G\in L^2(\R)$ (since $(t,x)$ is fixed, we have omitted the dependence of $G$ on it). Define
\[
\W_\eps(y)=\int_\R h_\eps(y-z)\W(dz)
\]
as a smooth mollification of $\dot{\W}$. Here $h_\eps(x)=\eps^{-1}h(x/\eps)$ with $h:\R\to\R_+$ smooth, even, compactly supported and satisfying $\int_\R h(x)dx=1$. We can define 
\[
v^\eps(t,x)=\int_\R G(y)\W_\eps(y)dy=\int_\R G\star h_\eps(z)\W(dz),
\]
and since $G\in L^2(\R)$, $G\star h_\eps\to G$ in $L^2(\R)$, so $v^\eps \to v$ in $L^2(\Omega)$ as $\eps\to0$. Now we show the $L^2(\Omega)$ limit of $v_\eps$ can also be written as a linear combination of $v_1,v_2,v_3$.

First we rewrite $v_\eps$ as
\[
\begin{aligned}
v_\eps(t,x)=&-\frac12\hat{R}(0)^\frac12\bar{a}^2\int_0^t \int_\R q_{\bar{a}(t-s)}(x-y)\partial_y(\partial_yu_{\hom}(s,y)\W_\eps(y))dyds\\
:=&(i)+(ii),
\end{aligned}
\]
with 
\begin{eqnarray*}
(i)&=&-\frac12\hat{R}(0)^\frac12\bar{a}^2\int_0^t \int_\R q_{\bar{a}(t-s)}(x-y)\partial_y^2u_{\hom}(s,y)\W_\eps(y)dyds,\\
(ii)&=&-\frac12\hat{R}(0)^\frac12\bar{a}^2\int_0^t \int_\R q_{\bar{a}(t-s)}(x-y)\partial_yu_{\hom}(s,y)\W_\eps'(y)dyds.
\end{eqnarray*}
It is clear that 
\[
(i)\to -\frac12\hat{R}(0)^\frac12\bar{a}^2v_3(t,x)
\]
in $L^2(\Omega)$. For $(ii)$, by the same proof as in Lemma~\ref{lem:spde} we have
\[
(ii)=-\frac12\hat{R}(0)^\frac12\bar{a}^2\E_W\{f'(x+\bar{\sigma}W_t)\int_0^t \W_\eps'(x+\bar{\sigma}W_s)ds\},
\]
and an application of It\^o's formula gives 
\[
\int_0^t \W_\eps'(x+\bar{\sigma}W_s)ds=\frac{2}{\bar{a}}\int_x^{x+\bar{\sigma}W_t}\W_\eps(y)dy-\frac{2}{\bar{a}}\int_0^t \W_\eps(x+\bar{\sigma}W_s)d\bar{\sigma}W_s.
\]

For the second part, we apply the duality relation in Malliavin calculus and the fact that the It\^o's integral is a particular case of the Skorohod integral \cite[Proposition 1.3.11]{nualart2006malliavin} to obtain
\[
\begin{aligned}
\E_W\{f'(x+\bar{\sigma}W_t)\int_0^t \W_\eps(x+\bar{\sigma}W_s)d\bar{\sigma}W_s\}=&\bar{a}\E_W\{f''(x+\bar{\sigma}W_t)\int_0^t \W_\eps(x+\bar{\sigma}W_s)ds\}\\
\to &\bar{a} v_3(t,x)
\end{aligned}
\]
in $L^2(\Omega)$.

For the first part, we write
\[
\int_x^{x+\bar{\sigma}W_t}\W_\eps(y)dy=\int_\R (1_{[x,x+\bar{\sigma}W_t]}\star h_\eps )(y)\W(dy),
\]
so it is clear that 
\[
\begin{aligned}
\E_W\{f'(x+\bar{\sigma}W_t)\int_x^{x+\bar{\sigma}W_t}\W_\eps(y)dy\}\to& \E_W\{f'(x+\bar{\sigma}W_t)\int_\R 1_{[x,x+\bar{\sigma}W_t]}(y)\W(dy)\}\\
=&v_1(t,x)-v_2(t,x)
\end{aligned}
\]
in $L^2(\Omega)$. The proof is complete.
\end{proof}

If we formally write in \eqref{eq:limitspde} that $\partial_x(\partial_xu_{\hom}\dot{\W})=\partial_x^2 u_{\hom}\dot{\W}+\partial_x u_{\hom}\ddot{\W}$, the term $\partial_x^2 u_{\hom}\dot{\W}$ does not come from $v_3$ since we have an opposite sign in \eqref{eq:v3}. If we recall that $v_1,v_2$ comes from the remainder $R_t^\eps$ and $v_3$ comes from the martingale $M_t^\eps$, this indicates that the errors coming from the martingale decomposition need to be rearranged to obtain the correct representation given by \eqref{eq:limitspde}.

\appendix

\section{Technical lemmas}
\label{s:lem}

\begin{lem}
$\E\{|\tilde{\phi}(x)|^2\}\les |x|$ and $\E\{|\tilde{\psi}(x)|^2\}\les |x|^3$.
\label{lem:mmpp}
\end{lem}

\begin{proof}
Since $\tilde{\phi}(x)=\bar{a}\int_0^x \tilde{V}(y)dy$ and $R(x)$ is the integrable covariance function of $\tilde{V}$, we have
\[
\E\{|\tilde{\phi}(x)|^2\}\les \int_0^x\int_0^x R(y-z)dydz \les |x|.
\]
For $\tilde{\psi}(x)$, by \eqref{eq:defpsi1} we have
\[
\tilde{\psi}(x)=-\frac{2}{\bar{a}}\int_0^x \tilde{\phi}(y)(\tilde{V}(y)+\bar{a}^{-1})dy=-2\int_0^x(\tilde{V}(y)+\bar{a}^{-1})\int_0^y \tilde{V}(z)dzdy
\]
so
\[
\begin{aligned}
&\E\{|\tilde{\psi}(x)|^2\}\\
\les &\int_0^x\int_0^x\int_0^{y_1}\int_0^{y_2}|\E\{(\tilde{V}(y_1)+\bar{a}^{-1})(\tilde{V}(y_2)+\bar{a}^{-1})\tilde{V}(z_1)\tilde{V}(z_2)\}|dz_1dz_2dy_1dy_2.
\end{aligned}
\]
In the above expression, we need to control the second, third and fourth moments of $\tilde{V}$, which is a mean-zero stationary random field of finite range dependence. For the term with the second moment, we have
\[
\int_0^x\int_0^x \int_0^{y_1}\int_0^{y_2}|R(z_1-z_2)|dz_1dz_2dy_1dy_2\les |x|^3.
\]
The other cases are discussed in the same way by applying Lemma~\ref{lem:mmes}.
\end{proof}

\begin{lem}[Moment estimates]
For any $x_i\in \R, i=1,2,3,4$, we have
\begin{equation}
|\E\{\prod_{i=1}^3 \tilde{V}(x_i)\}|\leq \rho(|x_1-x_2|)+\rho(|x_1-x_3|)+\rho(|x_2-x_3|)
\label{eq:3mm}
\end{equation}
and 
\begin{equation}
\begin{aligned}
|\E\{\prod_{i=1}^4 \tilde{V}(x_i)\}|\leq &\rho(|x_1-x_2|)\rho(|x_3-x_4|)+\rho(|x_1-x_3|)\rho(|x_2-x_4|)\\
&+\rho(|x_1-x_4|)\rho(|x_2-x_3|)
\end{aligned}
\label{eq:4mm}
\end{equation}
for some $\rho:\R_+\to\R_+$ satisfying $\rho(r)\les 1\wedge r^{-p}$ for any $p>0$.
\label{lem:mmes}
\end{lem}

\begin{proof}
Since $\tilde{V}$ is bounded, mean zero and of finite range dependence, \eqref{eq:4mm} comes from \cite[Lemma 3.1]{B-MMS-08}. For \eqref{eq:3mm}, it is clear that there exists a compactly supported $\rho:\R_+\to\R_+$ such that
\[
\begin{aligned}
|\E\{\prod_{i=1}^3 \tilde{V}(x_i)\}|\leq& \rho(\min\{ |x_1-x_2|,|x_1-x_3|,|x_2-x_3|\})\\
\leq &\rho(|x_1-x_2|)+\rho(|x_1-x_3|)+\rho(|x_2-x_3|).
\end{aligned}
\]
The proof is complete.
\end{proof}

\begin{lem}[Estimates on local time]
Let $L_t^x(y)$ be the local time of a standard Brownian motion $W_t$ starting from $x$ up to $t$, then for any $p\geq 1$, 
\[
\E\{|L_t^x(y)|^p\}\les t^{\frac{p}{2}}\int_{|y-x|}^\infty q_t(z)dz.
\]
\label{lem:localtime}
\end{lem}
\begin{proof}
First, $L_t^x(y)$ has the same distribution as $L_t^0(y-x)$. By the strong Markov property of Brownian motion and distribution property of $L_t^0(0)$, we further have 
\[
L_t^0(y-x)\sim L_{t-\tau_{y-x}}^0(0)1_{\tau_{y-x}\leq t}\sim M_{t-\tau_{y-x}}1_{\tau_{y-x}\leq t},
\] 
where $\tau_{y-x}$ is the hitting time of another independent Brownian motion starting at zero and reaching at $y-x$, and $M_t$ is the maximum of $W_t$ during $[0,t]$. Thus we have
\[
\begin{aligned}
\E\{|L_t^x(y)|^p\}=\int_0^t \E\{ |M_{t-s}|^p\} p^{\tau_{y-x}}(s)ds\les &t^{\frac{p}{2}}\int_0^t p^{\tau_{y-x}}(s)ds\\
=&t^{\frac{p}{2}} \P(\tau_{y-x}\leq t),
\end{aligned}
\]
with $p^{\tau_{y-x}}$ the density of $\tau_{y-x}$. The reflection principle tells that
\[
\P (\tau_{y-x}\leq t)=2\int_{|y-x|}^\infty q_t(z)dz.
\]
The proof is complete.
\end{proof}

\begin{lem}[SPDE representation]
Let $v(t,x)=\E_W\{ f(x+W_t)\int_0^t \dot{\W}(x+W_s)ds\}$, then it solves
\begin{equation}
\partial_t v(t,x)=\frac12\partial_{x}^2v(t,x)+ u(t,x)\dot{\W}(x)
\label{eq:spde}
\end{equation}
with zero initial condition, and the function $u$ solving $\partial_t u=\frac12\partial_{x}^2 u$ with initial condition $u(0,x)=f(x)$.
\label{lem:spde}
\end{lem}

\begin{proof}
The proof is similar to that of Proposition~\ref{p:spde}. First, we approximate the SPDE with a smooth equation. Then we use the probabilistic representation of the smooth equation and show its convergence.

The solution to \eqref{eq:spde} can be written as
\[
v(t,x)=\int_0^t \int_\R q_{t-s}(x-y)u(s,y)\W(dy)ds=\int_\R \left(\int_0^tq_{t-s}(x-y)u(s,y)ds\right) \W(dy),
\]
and we define $v_\eps(t,x)$ as
\[
v_\eps(t,x)=\int_0^t \int_\R q_{t-s}(x-y)u(s,y)\W_\eps(y)dyds=\int_\R \left(\int_0^tq_{t-s}(x-y)u(s,y)ds\right) \W_\eps(y)dy,
\]
with 
\[
\W_\eps(y)=\int_\R \frac{1}{\eps}h(\frac{x-y}{\eps})\W(dy)
\]
 as a smooth mollification of $\W$. It is clear that $v_\eps(t,x)\to v(t,x)$ in $L^2(\Omega)$ as $\eps\to 0$. 
 
Since $v_\eps$ solves the equation
\[
\partial_t v_\eps=\frac12\partial_x^2 v_\eps+u \W_\eps,
\]
by a probabilistic representation we can rewrite the solution as 
\[
v_\eps(t,x)=\E_W\{ \int_0^t u(t-s,x+W_s)\W_\eps(x+W_s)ds\}.
\]
Since $u$ solves the heat equation with initial condition $u(0,x)=f(x)$, we obtain
\[
\begin{aligned}
v_\eps(t,x)=&\E_W\E_B\{ \int_0^t f(x+W_s+B_{t-s})\W_\eps(x+W_s)ds\}\\
=& \E_W\{ f(x+W_t)\int_0^t \W_\eps(x+W_s)ds\}\\
=& \E_W\{ f(x+W_t)\int_\R \W_\eps(y) L_t^x(y)dy\},
\end{aligned}
\]
where $L_t^x(y)$ is the local time of $x+W_t$.	

By Lemma~\ref{lem:localtime}, for any $p\geq 1$, $\E\{|L_t^x(y)|^p\}$ can be bounded by some integrable function in $y$, and this helps to pass to the limit
\[
v_\eps(t,x)\to \E_W\{ f(x+W_t)\int_\R L_t^x(y)\W(dy)\}=\E_W\{ f(x+W_t)\int_0^t \dot{\W}(x+W_s)ds\}
\]
in $L^2(\Omega)$. The proof is complete.
\end{proof}

\end{document}